\documentclass{article}
\usepackage{amsfonts}
\usepackage{amssymb}
\usepackage{amsmath}
\usepackage{amsthm}
\usepackage{fullpage}
\usepackage{xcolor}
\usepackage{graphicx}
\usepackage{graphics}
\usepackage{pstricks}
\usepackage{pst-node}
\usepackage{pst-plot}
\usepackage{slashbox}
\usepackage{float}%
\allowdisplaybreaks
\newtheorem{theorem}{Theorem}

\newtheorem{lemma}{Lemma}
\newtheorem{proposition}{Proposition}

\newtheorem{example}{Example}
\newtheorem{remark}{Remark}
\newtheorem{notation}{Notation}

\newcommand{\rem}{{\mathrm{rem}}}

\newcommand{\hw}{{\mathrm{hw}}}

\begin{document}

\title{ On the number of terms of some families of the ternary\\ cyclotomic polynomials    $\Phi_{3p_2p_3}$ }
\author{Ala'a  Al-Kateeb \footnote{alaa.kateeb@yu.edu.jo} and Afnan Dagher\footnote{afnand@yu.edu.jo} \\
Department of Mathematics \\
Yarmouk University- Jordan 
}
\date{\today}
\maketitle

\begin{abstract}
We study the number of non-zero terms in two specific families of ternary cyclotomic polynomial, we find formulas for the number of terms by writing the cyclotomic polynomial as a sum of smaller sub-polynomials and study the properties of these polynomial. 
\end{abstract}
\textbf{Key Words}: Number of terms, Cyclotomic polynomials.

\textbf{2010 Mathematics Subject Classification}: 11B39, 11B83.

\section{Introduction}

The $n$-th cyclotomic polynomial $\Phi_{n}$ is defined as the monic polynomial
in $\mathbb{Z}[x]$ whose complex roots are the primitive $n$-th~roots of
unity. Due to its importance in many branches of mathematics, there have been extensive
investigation on its properties. Recently,  Sanna in \cite{SA},  write a  concise survey attempts to collect the main results regarding the coefficients of the cyclotomic polynomials and to provide all the relevant references to their proofs.
 
The investigation on height (maximum absolute value of coefficients) was
initiated by the finding that the height can be bigger than 1. 
It has produced numerous results, to list a few~\cite{AK1,BE3,BG1,BZ1,GA-Mo,GA-MO2,Mor2003,Mor2012,ED,KA1,KA2,ZH17}. 

The investigation on maximum gap (maximum difference between the consecutive exponents)  became a problem on its own because it could be viewed as a first step
toward understanding of sparsity structure of cyclotomic polynomials. In 2012
, it was shown   \cite{HLLP} that the maximum gap for binary cyclotomic
polynomial $\Phi_{p_{1}p_{2}}$ is $p_{1}-1$, that is, $g(\Phi_{p_{1}p_{2}})=
p_{1}-1$. In 2014, Moree~\cite{Moree2014} revisited the result and provided an
inspiring conceptual proof by making a connection to numerical semigroups of embedding dimension two. In 2016, Zhang~\cite{Zhang16} gave a simpler proof, along with
the result on the number of occurrences of the maximum gaps. In 2021, Al-Kateeb, et al. \cite{gap} proved that $g(\Phi_{mp})=\varphi(m)$, where $m$ is a square free odd integer and $p>m$ is prime number. 

The investigation of the number of non-zero terms in $\Phi_n(x)$ (also called the hamming weight $\hw(f)$)
was initiated in 1965 by Carlitz  \cite{CA1}, who gave an explicit formula for $\hw(\Phi_{pq})$, where $p<q$ are two distinct prime numbers.
 Hence, a  natural question one can ask; is there a formula for the number of terms of a ternary cyclotomic polynomials? that is the polynomial $\Phi_{pqr}$ with $p<q<r$ are three odd prime numbers, and ultimately,
arbitrary cyclotomic polynomials?
In 2014 Bezdega \cite{B2013} prove that the hamming weight of the cyclotomic plynomial $\Phi_n(x)$ is greater than or equal to $n^{\frac{1}{3}}$.  In 2016 \cite{AK} the number of terms for a ternary cyclotomic  polynomial was investigated  and following theorem was given (Theorem 7.1).   

 \begin{theorem} Let $p_1<p_2<p_3$ be odd prime numbers such that $p_2 \equiv \pm1 \mod p_1$ and $p_3 \equiv \pm 1 \mod p_1p_2$ and $p_3>p_1p_2$. Then 
\[\hw(\Phi_{p_1p_2p_3})= \begin{cases}N(p_3-1)+1, & r_3=1 \\
N(p_3+1)-1, & r_3=p_2-1 \\
\end{cases}\] where $N=\frac{2}{3}\left(\frac{(p_1-1)((p_1+4)(p_2-1)-(r_2-1))}{p_{1}p_2}\right)$ and $r_2 = p_2 \mod p_1, r_3= p_3 \mod p_1p_2.$
\end{theorem}

 It is natural to think about fixing the smallest prime number $p_1$ or to consider some small values of $r_2$ and $r_3$, since computations show that the structure of $\Phi_{p_1p_2p_3}$ is simpler for small values of $p_1, r_2$ or $r_3$. In this paper, we tackle the number of terms for of the cyclotomic polynomial $\Phi_{3p_2p_3}$ where $p_3 \equiv \pm 2 \mod 3p_2$.
We  come up with a nice formulas (Theorems \ref{thm:r2=1} and \ref{thm:r2=2}) for $\hw(\Phi_{3p_2p_3})$ where $p_3\equiv \pm 2 \mod 3p_2.$
 To prove the results we   use proof techniques used for studying
cyclotomic polynomials in \cite{AK}. \\
This paper is structerd as follows: In section \ref{s1} we list the main results of the paper, in section \ref{s2} we give several preliminaries and notations about cyclotomic polynomials needed for the rest of the paper. Then, in section \ref{s3} we prove the results of the paper. In appendix \ref{A1} we list some technical proofs and finally in appendix \ref{A2} we gave some examples that explain  the proof methods.
\section{Results}\label{s1}
Let $p_1< p_2$ and $p_3$ be  three  prime numbers  such that    $p_3>p_{1} \cdot p_{2}$. Throughout this paper, we denote%
\[
r_i:=\mathrm{rem}\left(  p_i,p_1\cdots p_{i-1}\right)  ,\;\;q_i:=\mathrm{quo}(p_i,p_1 \cdots p_{i-1})
\]
where $\mathrm{quo,rem}$ of course, stand for quotient and remainder.
\begin{theorem}\label{thm:r2=1}Let $3<p_2<p_3$ be odd prime numbers such that $p_2\equiv 1 \mod 3$ and $p_3>3p_2$. Then 
\[\hw(\Phi_{3p_2p_3})= \begin{cases}N(p_3-2)+\left(\frac{4p_2-1}{3}\right) , & \text{if}~~r_3=2 \\
N(p_3+2)-\left(\frac{4p_2-1}{3}\right), & \text{if}~~r_3=3p_2-2 \\
\end{cases}\] where $N= \frac{7(p_{2}^2-1)}{9p_2}$.
\end{theorem}
\begin{example}[Toy Example] In this example we use small prime numbers $p_2,p_3$ and use Theorem \ref{thm:r2=1} to compute $\hw(\Phi_{3p_2p_3})$. Let $p_2=7$, here $N=\frac{7 \cdot 48}{63}=\frac{16}{3}$.
\begin{itemize}
\item Let $p_3=23=1\cdot 3\cdot7+2\Rightarrow \hw(\Phi_{3\cdot 7\cdot 23}(x))= \frac{16}{3}(23-2)+ \frac{27}{3}=121$.
\item Let  $p_3=61=2\cdot 3\cdot7+19\Rightarrow \hw(\Phi_{3\cdot 7\cdot 61}(x))= \frac{16}{3}(61+2)+ \frac{27}{3}=327$.
\end{itemize} 
\end{example}
\begin{example}[Big Example]In this example we consider larger values of $p_2$ and $p_3$ which needs more time and effort to compute $\Phi_{3p_2p_3}(x).$ Let $p_2=283$, here $N=\frac{186872}{849}$. Let $p_3=84916133=100019\cdot 3\cdot283+2$. Then using Theorem \ref{thm:r2=1} we have \[ \hw(\Phi_{3\cdot283\cdot 84916133}(x))= \frac{186872}{849}(84916133-2)+ \frac{1131}{3}=18690750945\] 
\end{example}
\begin{theorem}\label{thm:r2=2}Let $3<p_2<p_3$ be odd prime numbers such that $p_2\equiv 2 \mod 3$ and $p_3>3p_2$. Then 
\[\hw(\Phi_{3p_2p_3})= \begin{cases}N(p_3-2)+\frac{4p_2+1}{3}, &\text{if}~~ r_3=2 \\
N(p_3+2)-\frac{4p_2+1}{3}, & \text{if}~~r_3=3p_2-2 \\
\end{cases}\] where $N=\frac {(p_2+1) ({7p_2-2})}{9p_2}. $
\end{theorem}
\begin{example}[Toy Example]In this example we use small prime numbers $p_2,p_3$ and use Theorem \ref{thm:r2=2} to compute $\hw(\Phi_{3p_2p_3})$. Let $p_2=5$, here $N=\frac{(5+1) \cdot (5\cdot 7-2)}{9\cdot 5}=\frac{22}{5}$.
\begin{itemize}
\item Let $p_3=17=3\cdot 3\cdot5+2\Rightarrow \hw(\Phi_{3\cdot 5\cdot 17}(x))= \frac{22}{5}(17-2)+ \frac{21}{3}=73$.
\item Let  $p_3=43=2\cdot 3\cdot5+13\Rightarrow \hw(\Phi_{3\cdot 5\cdot 13}(x))= \frac{22}{5}(43+2)-\frac{21}{3}=191$.
\end{itemize} 

\end{example}
\section{Preliminaries and Notations}\label{s2}
In this section we review some needed properties of  cyclotomic polynomials and  we give some important notations needed in the rest of the paper.
\subsection{Partition of Cyclotomic polynomials} In \cite{AHL17,AK}, a partition of cyclotomic  polynomials  was introduced, also the following properties were given. This partition can be used to simplify studying several properties of cyclotomic polynomial
\begin{notation}
[Partition]\label{notation:partition}Let%
\begin{align*}
\Phi_{mp}\left(  x\right)   &  =\sum_{i\geq0}f_{m,p,i}\left(  x\right)
\ x^{ip} &  &  \text{where }\deg f_{m,p,i}\left(  x\right)  <p\\
f_{m,p,i}\left(  x\right)   &  =\sum_{j\geq0}f_{m,p,i,j}\left(  x\right)
x^{jm} &  &  \text{where }\deg f_{m,p,i,j}\left(  x\right)  <m
\end{align*}

\end{notation}

\begin{notation}
[Operation]\label{not:operation} For a polynomial $f$ of degree less than $m$, let

\begin{enumerate}
\item $\mathcal{T}_{s}f={\mathrm{rem}}(f,x^{s})$ \hspace{9em}
\textquotedblleft Truncate\textquotedblright

\item $\mathcal{F}f=x^{m-1}f\left(  x\newline^{-1}\right)  $ \hspace{8em}
\textquotedblleft Flip\textquotedblright

\item $\mathcal{R}_{s}f={\mathrm{rem}}(x^{m-\operatorname*{rem}\left(
s,m\right)  }f,x^{m}-1)$ \hspace{1.6em} \textquotedblleft
Rotate\textquotedblright

\item $\mathcal{E}_{s}f=f(x^{\operatorname*{rem}\left(  s,m\right)  })$
\hspace{8.3em} \textquotedblleft Expand\textquotedblright
\end{enumerate}
\end{notation}
\noindent Throughout this paper, for an integer $m$ and a prime $p$, we denote%
\[
r:=\mathrm{rem}\left(  p,m\right)  ,\;\;q:=\mathrm{quo}(p,m)
\]
where $\mathrm{quo,}$ of course, stand for quotient.
The formula of the  sub-polynomials $f_{m,p,i,j}$ is given by the following theorem. 
\begin{theorem}
[Block]\label{thm:explicit} For $0\leq i\leq\varphi(m)-1$ and $0\leq j\leq
q$,
\[
f_{m,p,i,j}=%
\begin{cases}
-\mathcal{R}_{ir}(\Psi_{m}\cdot\mathcal{E}_{r}\mathcal{T}_{i+1}\Phi_{m}) &
0\leq j\leq q-1\\
~~\mathcal{T}_{r}f_{m,p,i,0} & j=q
\end{cases}
\]
where $\Psi_{m}\left(  x\right)  =\frac{x^{m}-1}{\Phi_{m}\left(  x\right)  }$,
the $m$-th inverse cyclotomic polynomial.
\end{theorem}
For more information and properties of $\Psi_m(x)$ see \cite{Mor09}.

\begin{notation}
We will also use the following notations%
\[
\Phi_{m}=\sum_{s\geq0}a_{s}x^{s}\ \ \ \ \ \ \ \Psi_{m}=\sum_{t\geq0}b_{t}%
x^{t}\ \ \ \ \ \ \ \ f_{m,p,i,0}=\sum_{k=0}^{m-1}c_{k}x^{k}%
\]

\end{notation}

\begin{lemma}
[Explicit expression for blocks]\label{l:block}Let $p>m$. Let
$r=\operatorname*{rem}\left(  p,m\right)  $. Then the blocks $f_{m,p,i,j}$ can
be written explicitly as in

\begin{enumerate}
\item $f_{m,p,i,j}=\left\{
\begin{array}
[c]{ll}%
-\sum\limits_{s=0}^{i}a_{s}\mathcal{R}_{\left(  i-s\right)  r}\Psi_{m} &
\text{if }j<q\\
\mathcal{T}_{r}f_{m,p,i,0} & \text{if }j=q
\end{array}
\right.  $

\item $c_{k}=-\sum\limits_{s=0}^{i}a_{s}b_{\operatorname*{rem}\left(
k+\left(  i-s\right)  r,m\right)  }$

\end{enumerate}
\end{lemma}

\begin{proof}
\ 

\begin{enumerate}
\item Note%
\[
\Phi_{mp}=\frac{\Phi_{m}\left(  x^{p}\right)  }{\Phi_{m}\left(  x\right)
}=-\Phi_{m}\left(  x^{p}\right)  \ \Psi_{m}\ \frac{1}{1-x^{m}}=-\ \Phi
_{m}\left(  x^{p}\right)  \ \Psi_{m}\ \sum_{u\geq0}x^{um}=\sum_{s\geq0}%
x^{sp}\ \left(  -\ a_{s}\Psi_{m}\ \sum_{u\geq0}x^{um}\right)
\]
Thus $\Phi_{mp}$ is the sum of weighted-shifted $\Psi_{m}$, as illustrated by
the following diagram.
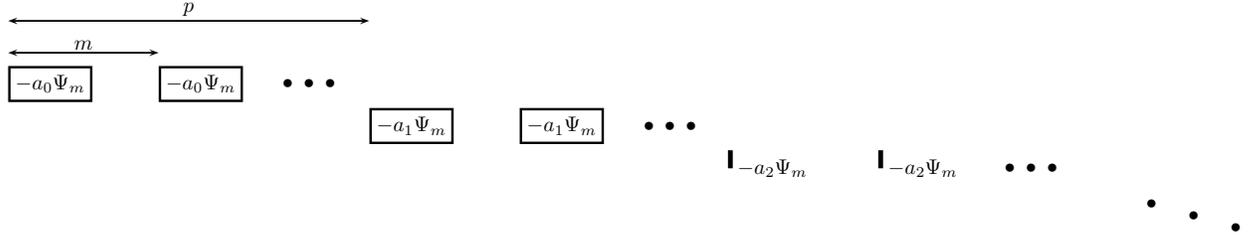
\[
\scalebox{0.80}{
\begin{pspicture}(0,-2.4)(21,1.2)
\psline[linewidth=0.03cm,arrows=<->](0,1.05)(6,1.05)
\rput(3,1.25){$p$}
\psline[linewidth=0.03cm,arrows=<->](0,0.52)(2.5,0.52)
\rput(1.25,0.65){$m$}
\def\polya{
\psframe[linewidth=0.04,dimen=outer](0.0,0.3)(1.4,-0.3)
\rput(0.7,0){$-a_0\Psi_m$}
}
\multirput(0,0)(2.5,0){2}{\polya}
\rput(5,0){$\bullet\;\bullet\;\bullet$}
\def\polyb{
\psframe[linewidth=0.04,dimen=outer](0.0,0.3)(1.4,-0.3)
\rput(0.7,0){$-a_1\Psi_m$}
}
\multirput(6,-0.7)(2.5,0){2}{\polyb}
\rput(11,-0.7){$\bullet\;\bullet\;\bullet$}
\def\polyc{
\psframe[linewidth=0.04,dimen=outer](0.0,0.3)\rput(0.7,0){$-a_2\Psi_m$}
}
\multirput(12,-1.4)(2.5,0){2}{\polyc}
\rput(17,-1.4){$\bullet\;\bullet\;\bullet$}
\multirput(19,-2)(0.7,-0.2){3}{$\bullet$}
\end{pspicture}}
\]
The claim is immediate from the $f_{m,p,i,j}$ slice of the above diagram.

\item Note%
\[
f_{m,p,i,0}=-\sum\limits_{s=0}^{i}a_{s}\mathcal{R}_{\left(  i-s\right)  r}%
\sum_{k\geq0}b_{k}x^{k}=-\sum\limits_{s=0}^{i}a_{s}\sum_{k\geq0}%
b_{\operatorname*{rem}\left(  k+\left(  i-s\right)  r,m\right)  }x^{k}%
=\sum_{k\geq0}\left(  -\sum\limits_{s=0}^{i}a_{s}b_{\operatorname*{rem}\left(
k+\left(  i-s\right)  r,m\right)  }\right)  x^{k}%
\]
Thus, $c_{k}=-\sum\limits_{s=0}^{i}a_{s}b_{\operatorname*{rem}\left(
k+\left(  i-s\right)  r,m\right)  }$

\end{enumerate}
\end{proof}

\begin{theorem}
[Intra-Structure]\ \label{thm:intra-str} Within a cyclotomic polynomial, we have

\begin{enumerate}
\item (Repetition)\ \ $f_{m,p,i,0}=\cdots=f_{m,p,i,q-1}$

\item (Truncation) $f_{m,p,i,q}=\mathcal{T}_{r}f_{m,p,i,0}$

\item (Symmetry)\ $f_{m,p,i^{\prime},0}=\mathcal{R}_{\varphi(m)-1-r}%
\mathcal{F}f_{m,p,i,0}$ \hspace{2.4em} if $i^{\prime}+i=\varphi(m)-1$
\end{enumerate}
\end{theorem}

\begin{theorem}
[Inter-Structure]\ \label{thm:inter-str} Among cyclotomic polynomials, we have

\begin{enumerate}
\item (Invariance)\ \ \ \ \ \ \ \ \ $f_{m,\tilde{p},i,0}=~~f_{m,p,i,0}$
\hspace{4.8em} if $\tilde{p}-p\equiv 0 \mod m$

\item (Semi-Invariance)\ \ $f_{m,\tilde{p},i,0}=-\mathcal{R}_{\varphi
(m)-1}\mathcal{F}f_{m,p,i,0}\;\;$ if $\tilde{p}+p\equiv0 \mod m$
\end{enumerate}
\end{theorem}
As an application of the properties above we  have the following theorem.
\begin{theorem}[Hamming weight]\label{ex:hw}
Let $\hw(f)$ stands for the number of non-zero terms in the polynomial $f$. Then we have 
\begin{enumerate}
\item ${[Linear]}$ $\ \ \mathrm{hw}(\Phi_{mp})=A\cdot p+B$

\item ${[Parallel]}$ $\ \ \mathrm{hw}(\Phi_{m \tilde{p}})=A\cdot  \tilde{p}-B$
\end{enumerate} where $p+\tilde{p} \equiv 0 \mod m$.
\end{theorem}
In \cite{gap}, another  version of $f_{m,p,i,0}$  was also given.

\subsection{Properties of $\Phi_{3p_2p_3}$}
In this subsection we introduce some properties of $\Phi_{3p_2p_3}(x)$ to be used in the proofs of the main results.
\begin{lemma}
\label{lemma:Phi3p} Let $m=3p_{2}, p_{2}>3$ be a prime number. Let $\Phi_{m}=\sum_{s}%
a_{s}x^{s}$. For $0\leq i\leq p_{2}-1$, we have
\[
a_{i}=\left\{
\begin{array}
[c]{ll}%
1, & \text{if } i \equiv0 \mod 3\\
-1, & \text{if } i \equiv1 \mod 3\\
0, & \text{if } i \equiv2 \mod 3
\end{array}
\right.
\]
and for $p_2\leq i\leq \varphi(3p_2)$, we have
\[
a_{i}=\left\{
\begin{array}
[c]{ll}%
0, & \text{if } i \equiv0 \mod 3\\
-1, & \text{if } i \equiv1 \mod 3\\
1, & \text{if } i \equiv2 \mod 3
\end{array}
\right.
\]

\end{lemma}
\begin{proof}Immediate from Notation \ref{notation:partition}, Theorem \ref{thm:explicit} or Lemma \ref{l:block}. We moved the proof into Appendix \ref{A1}
\end{proof}
\begin{lemma}\label{lem:tr11}For  $i=1,\cdots,p_2-2$, we have \[
\mathcal{T}_{i+1} \Phi_{3p_2} = \left\{
\begin{array}
[c]{ll}%
  (1-x)\sum_{j=0}^{\left\lfloor \frac{i-1}{3} \right\rfloor } x^{3j} +x^{i},& \text{if } i \equiv0 \mod 3\\
(1-x)\sum_{j=0}^{\left\lfloor \frac{i}{3} \right\rfloor} x^{3j},  & \text{if } i \equiv1, 2 \mod 3\\

\end{array}
\right.
\]

\end{lemma}\begin{proof}Immediate from Lemma \ref{lemma:Phi3p}.
\end{proof}
\begin{lemma}\label{lem:tr111}For  $i=p_2-1,\cdots,\varphi(3p_2)-2$ and $p_2\equiv 1 \mod 3$, we have  \[\mathcal{T}_{i+1} \Phi_{3p_2} =  (1-x)\sum_{j=0}^{q_2-1 } x^{3j}+(1-x^2)x^{p_2-1} \left\{
\begin{array}
[c]{ll}%
\sum_{j=0}^{\lfloor \frac{i-p_{2}-1}{3}\rfloor }x^{3j} +x^{i},& \text{if } i \equiv0 \mod 3\\
\sum_{j=0}^{\lfloor \frac{i-p_{2}-1}{3}\rfloor }x^{3j} +x^{i-1},& \text{if } i \equiv1,2 \mod 3\\
\end{array}
\right.
\] 
\end{lemma}\begin{proof}Immediate from Lemma \ref{lemma:Phi3p}.
\end{proof}

\begin{proposition}\label{idea1} For $0\leq i \leq \varphi(3p_2)-1$ we have
 \begin{align*}
 f_{m,p,i,0}& = \mathcal{N}\mathcal{R}_{\rem(2i,3p_2)} \left(\Psi_{3p_2} \mathcal{E}_2\mathcal{T}_{i+1} \Phi_{3p_2}\right)\\
&= -\mathcal{R}_{\rem(2i,3p_2)} \left(\Psi_{3p_2} \mathcal{E}_2\mathcal{T}_{i+1} \Phi_{3p_2}\right)&\\
&=- \rem\left(x^{3p_2-\rem(2i,3p_2)}C_i,x^{3p_2}-1 \right)\\
\end{align*}
where \(C_i=\Psi_{3p_2}   \begin{cases}(x^2-1)\sum_{j=0}^{\left\lfloor \frac{i-1}{3} \right\rfloor } x^{6j} -x^{2i},& \text{if } i \equiv0 \mod 3 \\
(x^2-1)\sum_{j=0}^{\left\lfloor \frac{i}{3} \right\rfloor} x^{6j},  & \text{if } i \equiv1, 2 \mod 3\\ 
\end{cases}\)
\\ if $i<p_2-1\Rightarrow2i=2p_2-2<3p_2=m\Rightarrow \rem(2i,3p_2)=2i$
\end{proposition}

\begin{remark}\label{rem:1} The number of terms will not be changed with rotation and negation, so in order   to study the number of terms of  $f_{m,p,i,0}$ it is enough to study the number of terms of $C_{i}$.
\end{remark}

\section{Proofs}\label{s3}In this section we will prove the results of this paper. 
 \subsection{Proof of Theorem \pageref{thm:r2=1}($r_2=1$) } \label{sec:r2=1}In this subsection we assume that $r_2=1$ that is $p_2\equiv 1 \mod 3,$ and   $r_3\equiv 2 \mod 3p_2$. In the following  six Lemmas we will compute $\hw(f_{3p_2,p_3,i,0})$ for several values of $i$, in order to be used later in proving the main result.
 
\begin{lemma}\label{lem:not mult31} Let $i=3u+v<p_2-1$, where $v\in \{1,2\}$. Then \[\hw(f_{3p_2,p_3,i,0})=\begin{cases}8(u+1),~~~~~~ ~~ & ~~\text{if}~~~u=0,1,\cdots, \frac{q_2}{2}-1  \\
4(u+1+ \frac{q_2}{2} ), &~~\text{if}~~ u= \frac{q_2}{2},  \cdots,  q_2-1\\
\end{cases}\]
\end{lemma}
\begin{proof}As stated in Remark  \ref{rem:1} we have  
\begin{align*}\hw(f_{3p_2,p,i,0})&=\hw(C_i)\\&= \hw\left(\Psi_{3p_2}   
\cdot (x^2-1)\cdot\sum_{j=0}^{\left\lfloor \frac{i}{3} \right\rfloor} x^{6j}\right)& \text{ by Lemma \ref{lem:tr11}}\\
&=\hw\left(\Phi_3(x)\cdot (x^{p_2}-1)   
(x^2-1)\sum_{j=0}^{\left\lfloor \frac{i}{3} \right\rfloor} x^{6j}\right)\\ &=\hw\left((1+x)(x^{p_2}-1)   
(x^3-1)\sum_{j=0}^{\left\lfloor \frac{i}{3} \right\rfloor} x^{6j}\right)& \Phi_3(x)\cdot(x^2-1)=(x+1)(x^3-1)\\ &=\hw\left((1+x)(1-x^3)\left(\sum_{j=0}^{u} x^{6j}-\sum_{j=0}^{u} x^{6j+p_2}\right)\right)\\
&=2 \cdot  \hw\left(\sum_{j=0}^{u} x^{6j}-\sum_{j=0}^{u} x^{6j+3}-\sum_{j=0}^{u} x^{6j+p_2}+\sum_{j=0}^{u} x^{6j+p_2+3}\right)\end{align*}
If $u \leq \frac{q_2}{2}-1$, then there isn't  any cancellation in the  above sum , thus  $\hw(f_{3p_2,p_3,i,0})=2\cdot 4 \cdot (u+1)=8(u+1)$ as desired. On the other hand, if $ \frac{q_2}{2} \leq u \le q_2-1$, then we have  \begin{align*}\hw(f_{3p_2,p_3,i,0})&=    \hw\left((x+1)(1-x^{p_2})\sum_{j=0}^{u} x^{6j}+(x+1)x^3 (x^{p_2}-1)\sum_{j=0}^{u} x^{6j}\right)\\&=  \hw\left((x+1)(1-x^{p_2})\sum_{j=0}^{u} x^{6j} \right)+\hw \left(x^3(x+1) (x^{p_2}-1)\sum_{j=0}^{u} x^{6j}\right)\\
&= 2 \cdot   \hw\left((x+1)(1-x^{p_2})\sum_{j=0}^{u} x^{6j} \right)\\&=2 \cdot   \left( \hw\left(A_i \right)+ \hw \left(B_i\right)\right)
\end{align*}Where $A_i\equiv A_u=(1-x^{p_2+1})\sum_{j=0}^{u} x^{6j}$ and $B_i\equiv B_u=(x-x^{p_2})\sum_{j=0}^{u} x^{6j} $.
We will study $A_i$ and $B_i$ for the case $r_2=1$.
 we claim that $\hw(A_i)=2(u+1)$ and $\hw(B_i)=q_2$. For $A_u$, if there is any cancelation in the sum, then $p_2+1+6j=6k$ for some integers  $j$ and $k$. Thus, $p_2+1 \equiv  0 \mod 3$ which contradicts the fact that $p_2\equiv 1 \mod 3.$ Now, we claim that  $B_u= \sum_{j=0}^{\frac{q_2}{2}-1}x^{6j+1}~- x^{p_2+6(1+j+u-\frac{q_2}{2})}$
We prove the claim by induction on $u$ starting from $u=\frac{q_2}{2}$.  \begin{itemize}
\item If $u=\frac{q_2}{2}$, then \begin{align*}B_u&=(x-x^{p_2})\sum_{j=0}^{\frac{q_2}{2}}x^{6j}\\
&=x+x^7+ \cdots+x^{3q_2+1}-x^{p_2}-x^{p_2+6}-\cdots-x^{p_2+3q_2}\\
&=\sum_{j=0}^{\frac{q_2}{2}-1}x^{6j+1}~- x^{p_2+6(1+j)} \end{align*}
\item\ Assume that $B_u=\sum_{j=0}^{\frac{q_2}{2}-1}x^{6j+1}~- x^{p_2+6(1+j+u-\frac{q_2}{2})}$
\item Consider \begin{align*}B_{u+1}&=(x-x^{p_2})\sum_{j=0}^{u+1} x^{6j}\\
                                    &=(x-x^{p_2})\sum_{j=0}^{u} x^{6j}+x^{6u+6}(x-x^{p_2})\\
                                    &=\sum_{j=0}^{\frac{q_2}{2}-1}(x^{6j+1}~- x^{p_2+6(1+j+u-\frac{q_2}{2})})+~ (x^{6u+7}-x^{p_2+6u+6})\\
&= \sum_{j=0}^{\frac{q_2}{2}-1}x^{6j+1}~- \sum_{j=1}^{\frac{q_2}{2}-1} x^{p_2+6(1+j+u-\frac{q_2}{2})}-x^{6u+7}+~ (x^{6u+7}-x^{p_2+6u+6})&\text{since}~~ p_2-3q_2=1\\
&=\sum_{j=0}^{\frac{q_2}{2}-1}x^{6j+1}~- \sum_{j=1}^{\frac{q_2}{2}-1} x^{p_2+6(1+j+u-\frac{q_2}{2})}-x^{p_2+6u+6}\\
&=\sum_{j=0}^{\frac{q_2}{2}-1}x^{6j+1}~- \sum_{j=0}^{\frac{q_2}{2}-1} x^{p_2+6(2+j+u-\frac{q_2}{2})}  & \text{by reindexing} 
   \end{align*}
\end{itemize}by induction we have $B_u=\sum_{j=0}^{\frac{q_2}{2}-1}x^{6j+1}~- x^{p_2+6(1+j+u-\frac{q_2}{2})}$ as desired. 
 Thus $\hw(f_{3p_2,p_3,0})=4(u+1+ \frac{q_2}{2} ),$ for $\frac{q_2}{2}   \leq u \leq q_2-1$.\end{proof}
\begin{lemma} \label{lem:rem1}If $i=3u+v <p_2-1$ where $v\in \{1,2\}$, then\[\hw(f_{3p_2,p,i,q})=\begin{cases}1, &\text{if}~~ u=0,1,\cdots, \frac{q_2}{2}-1  \\
3-v, &\text{if}~~ \frac{q_2}{2} \leq u \le q
_2-1
\end{cases}\]
\end{lemma}
\begin{proof}
It is enough to show that \textbf{only one} of  the terms $x^0$ and $x^1$ will appear in $f_{3p_2,p,i,0}$ when $0\leq u \leq \frac{q_2}{2}-1$ and \textbf{only one} of the terms $x^0$ and $x$ will appear in  $f_{3p_2,p,i,0}$ when $ \frac{q_2}{2}\leq u \leq q_2$ if $v=2$, while both of them will appear when $v=1.$ In the table below we list the terms  that appear in $C_i=\Psi_{3p_2}   
\cdot (x^2-1)\cdot\sum_{j=0}^{  u} x^{6j}$ and the corresponding terms in $f_{3p_2,p,i,0}=\rem(x^{-2i}C_i, x^{3p_2}-1)$.
\[\begin{array}{c|c|c|c|c}
u/\text{Terms in }  & C_i  ~~ \text{if}~~ v=1 & ~ f_{3p_2,p,i,0} ~\text{if } v=1&  C_i ~~ \text{if}~~ v=2 & ~ f_{3p_2,p,i,0} ~\text{if } v=2 \\\hline
0\leq u \leq \frac{q_2}{2}-1 & -x^{6u+2v+1} &-x&-x^{6u+2v}& -1 \\
\frac{q_2}{2} \leq u \leq q_2-1& -x^{6u+2v}-x^{6u+2v+1} & -1-x &x^{6u+2v+1}&x\\
\end{array}\] For the case  $0 \leq u \leq \frac{q_2}{2}-1 $, recall   the proof  of~ Lemma \ref{lem:not mult31}, we have  $C_i=(1+x)(1-x^3)\left(\sum_{j=0}^{u} x^{6j}-\sum_{j=0}^{u} x^{6j+p_2}\right)$. So it is clear that $x^{6u+2v}$ will not appear when $v=1$, while $x^{6u+2v+1}$ will not appear when $v=2.$\\
 Now for the case $u > \frac{q_2}{2}-1$, again from  $C_i=(1+x)(1-x^3)\left(\sum_{j=0}^{u} x^{6j}-\sum_{j=0}^{u} x^{6j+p_2}\right)$. When $v=1$ the terms $-x^{6u+3}-x^{p_2+6(u-\frac{q_2}{2})+1}=-x^{6u+2v}-x^{6u+2v}$ will appear. Finally, when $v=2$ we have only the term $x^{p_2+6(u-\frac{q_2}{2})+4}=x^{6u+5}=x^{6u+2v+1}$.
\end{proof}

\begin{lemma}\label{lem:multof31} If $i=3u<p_2-1$, then \[\hw(f_{3p_2,p_3,i,0})=\begin{cases}8u+6, &  \text{if} ~~u=0,1,\cdots,\frac{q_2}{2} -1 \\
4u+2q_2+2,  & \text{if}~~ u= \frac{q_2}{2}, \cdots , q_2 \\
\end{cases}\]
\end{lemma}
\begin{proof}The case when $u=0$ is trivial. Consider  $0< u\leq q_2$, in this case the cyclotomic polynomial $\Phi_{3p_2p_3}$ is  flat (see Theorem 38 in  \cite{ED}), so we will only  worry about cancelations. From Remark  \ref{rem:1} we have
\begin{align*}\hw(f_{3p_2,p_3,i,0})&= \hw(C_i)\\
                                   &=\hw\left(\Psi_{3p_2}   
\cdot\left( (x^2-1)\cdot\sum_{j=0}^{\left\lceil \frac{i-1}{3} \right\rceil} x^{6j}~-~x^{2i}\right)\right)\\
&= \hw \left(\Psi_{3p_2}   
\cdot \left((x^2-1)\cdot\sum_{j=0}^{\left\lceil  u-\frac {1}{3} \right\rceil} x^{6j}~-~x^{2i}\right) \right)\\
&= \hw \left(\Psi_{3p_2}   
\cdot (x^2-1)\cdot\sum_{j=0}^{  u-1} x^{6j}-x^{6u}\Psi_{3p_2} \right)
\end{align*}
We handle two cases:\begin{itemize}
\item 
$1 \leq u \leq \frac{q_2}{2} -1$: In this case there are no cancelations  between the sums $S_1=\Psi_{3p_2}   
\cdot (x^2-1)\cdot\sum_{j=0}^{  u-1} x^{6j} $ and  $S_2=x^{6u}\Psi_{3p_2}   $. Notice that $S_1= (1-x^{p_2})(1+x-x^3-x^4)\sum_{j=0}^{u-1}x^{6j}$ and $S_2=\Phi_3\cdot(x^{p_2+6u}-x^{6u})$. Thus $\hw(f_{3p_2,p_3,i,0})=8u+6.$
\item $u \geq \frac{q_2}{2} $:\ Notice that $\Psi_{3p_2}\cdot (x^2-1)=(x+1)(x^3-1)(1-x^{p_2})$, so we have
\[\hw(f_{3p_2,p_3,i,0})=  \hw \left((x+1)A_u+B_u\right)\] where  $A_u=(x^3-1)(1-x^{p_2})\sum_{j=0}^{u-1}x^{6j}$and $B_u=x^{6u}\Psi_{3p_2}   
$. We will use mathematical induction on $u$ to prove that  \[(x+1)A_u=(x+1)\sum_{j=0}^{q_2-1}(-1)^{j+1}x^{3j}+ (1-x^2)\sum_{j=q_2}^{2u-1}(-1)^{j+1}x^{3j}+(x+1)\sum_{j=2u}^{2u+q_2-1}(-1)^jx^{3j+1}\]  \begin{itemize}
\item For $u=\frac{q_2}{2}:$
\begin{align*}(x+1)A_{\frac{q_2}{2}}&= (x+1)(x^3-1)(1-x^{p_2})\sum_{j=0}^{\frac{q_2-2}{2}}x^{6j}\\
                                & = (x+1)(1-x^{p_2}) \sum_{j=0}^{q_2-1}(-1)^{j+1}x^{3j}\\
         &=(x+1) \left(\sum_{j=0}^{q_2-1}(-1)^{j+1}x^{3j}+\sum_{j=0}^{q_2-1}(-1)^jx^{3j+p_2}\right)\\
 &=(x+1) \left(\sum_{j=0}^{q_2-1}(-1)^{j+1}x^{3j}+\sum_{j=q_{2}}^{2q_2-1}(-1)^jx^{3j+1}\right)       
\end{align*}
\item Assume that  \[(x+1)A_u=(x+1)\sum_{j=0}^{q_2-1}(-1)^{j+1}x^{3j}+ (1-x^2)\sum_{j=q_2}^{2u-1}(-1)^{j+1}x^{3j}+(x+1)\sum_{j=2u}^{2u+q_2-1}(-1)^jx^{3j+1}\]
\item Consider \begin{align*}(x+1)A_{u+1}&= (x+1)(x^3-1)(1-x^{p_2})\sum_{j=0}^{u}x^{6j}\\
                                    &= (x+1)(x^3-1)(1-x^{p_2})\sum_{j=0}^{u-1}x^{6j}+(x+1)(x^3-1)(1-x^{p_2})x^{6u}\\
                                    &= (x+1)\sum_{j=0}^{q_2-1}(-1)^{j+1}x^{3j}+ (1-x^2)\sum_{j=q_2}^{2u-1}(-1)^{j+1}x^{3j}\\&+(x+1)\sum_{j=2u}^{2u+q_2-1}(-1)^jx^{3j+1}& \text{by induction}\\
 &   +(-x^{6u}-x^{6u+1}+x^{6u+3}+x^{6u+4}+x^{6u+p_2}+x^{6u+p_2+1}-x^{6u+p_2+3}-x^{6u+p_2+4})\\ 
 &= (x+1)\sum_{j=0}^{q_2-1}(-1)^{j+1}x^{3j}+ (1-x^2)\sum_{j=q_2}^{2u+1}(-1)^{j+1}x^{3j}\\&+(x+1)\sum_{j=2u+2}^{2u+q_2+1}(-1)^jx^{3j+1}& \\
\end{align*}as desired.
\end{itemize}
Now, it remains to find $\hw(f_{3p_2,p,i,0})=  \hw \left((x+1)A_u+B_u\right)$. Notice that there are two cancelations between $B_u$ and  the third sum  from $(x+1)A_u$, namely  $x^{6u+1}$ and $x^{6u+2}$. Thus $\hw(f_{3p_2,p,i,0})=  \hw \left((x+1)A_u+B_u\right)$ equals the sum of hamming weights of the three sums of $(x+1)A_u$ plus 2. so $\hw \left((x+1)A_u+B_u\right)=2+2q_2+2(2u-q_2)+2q_2=4u+2q_2+2$. 
\end{itemize}
\end{proof}
\begin{lemma}\label{lem:rem2} If $i=3u$, then\[\hw(f_{3p_2,p,i,q_{}})=\begin{cases}2, ~~\text{if} &~ u=0,1,\cdots, \frac{q_2}{2} -1 \\
1, ~~\text{if} &  u= \frac{q_2}{2}, \cdots,  q_2 \\
\end{cases}\]
\end{lemma}
\begin{proof}It is enough to prove that both the terms $x^0$ and $x$ will appear in $f_{3p_2,p,i,0}$ when $0\leq u \leq \frac{q_2}{2}-1$ and only one of the terms $x^0$ and $x$ will appear in  $f_{3p_2,p,i,0}$ when $ \frac{q_2}{2}\leq u \leq q_2.$ In the table below we list the term appears in $C_i=\Psi_{3p_2}   
\cdot (x^2-1)\cdot\sum_{j=0}^{  u-1} x^{6j}-x^{6u}\Psi_{3p_2}$ and the corresponding term in $f_{3p_2,p,i,0}$.
\[\begin{array}{c|c|c}
u & \text{Terms in}~ C_i & \text{ Terms in }~ f_{3p_2,p,i,0}  \\\hline
0\leq u \leq \frac{q_2}{2}-1 & x^{6u}+x^{6u+1} &1+x  \\
\frac{q_2}{2} \leq u \leq q_2& x^{6u} & 1 \\
\end{array}\] The reasoning is clear from the proof of Lemma \ref{lem:multof31}.
\end{proof}
\begin{lemma}\label{lem:p2-1}$\hw(f_{3p_2,p_3,p_2-1,0})=\begin{cases}2(p_2-1), & \text{if}~  r_2=1\\
2p_2+1, & \text{if~} r_2=2 \\
\end{cases}$
\end{lemma}\begin{proof}Immediate using Remark  \ref{rem:1} and analogous arguments in the proofs of other lemmas. \end{proof}
\begin{lemma}\label{lem:trunc} If $i\geq p_2$ , then $\hw(f_{3p_2,p,i,q})=0$
\end{lemma}
\begin{proof} From Lemma \ref{l:block}-2 we know that  \[c_k =-\sum_{s=0}^i a_sb_{\rem(k+2(i-s), 3p_2)}
\] 
and for $\Psi_{3p_2}$ \[b_k=\begin{cases}-1 & k=0,1,2 \\
1 & k=p_2,p_2+1,p_2+2 \\
0& otherwise \\
\end{cases}\]
 We will compute $c_0$ and $c_1$ for each $i \geq p_2.$ 
 At first 
 \begin{align*}c_0&= -\sum_{s=0}^i a_sb_{2(i-s)}\\
                  &=-(a_{i-\frac{p_2+1}{2}}b_{p_2+1}+a_{i-1}b_{2}+a_{i}b_0)& \text{the only possible values for}~ k~ \text{to be even and}~ b_k \neq 0 \\
                  &= -(a_{i-\frac{p_2+1}{2}}-a_{i-1}-a_{i}) \\
                  &=0
 \end{align*} we summarise the last conclusion in the next table, we use Lemma \ref{lemma:Phi3p} and the fact that $i \equiv i-\frac{p_2+1}{2} \mod 3$ and $i> p_2.$\[ \begin{array}{c|c|c|c|c}
i \mod 3 & a_i & a_{i-1} & a_{i-\frac{p_2+1}{2}}& a_{i-\frac{p_2+1}{2}}-a_{i-1}-a_{i}\\\hline
0 & 0 & 1 & 1&0 \\
1 & -1 & 0 & -1 &0\\
2 & 1 & -1 & 0&0 \\
\end{array}\]
For $i=p_2$ and $p_2\equiv 1 \mod 3 $ we have $a_{i-\frac{p_2+1}{2}}-a_{i-1}-a_{i}=-1-0+1=0$  and if $p_2\equiv 2 \mod 3$ we have $a_{i-\frac{p_2+1}{2}}-a_{i-1}-a_{i}=0+1-1=0$. So from the discussion above we see that $c_0=0.$ It remains to prove that $c_1=0$ which can be done  using similar ideas so we skip the proof of that part.

 \end{proof}

\begin{theorem}Let $3<p_2<p_3$ be odd prime numbers such that $p_2\equiv 1 \mod 3$. Then 
\[\hw(\Phi_{3p_2p_3})= \begin{cases}N(p_3-2)+\left(\frac{4p_2-1}{3}\right) ,& \text{if~~}r_3=2 \\
N(p_3+2)-\left(\frac{4p_2+1}{3}\right), &\text{if~~} r_3=3p_2-2 \\
\end{cases}\] where $N= \frac{7(p_{2}^2-1)}{9p_2^2}$.
\end{theorem}
\begin{proof}For $r_3=2$, note
\begin{align*}\hw(\Phi_{3p_2p_3})&=  \sum_{\substack{0\leq i\leq\varphi
(m)-1}}{\mathrm{hw}}(f_{m,p,i})\ \ x^{ip_{3}}\\
&  =q_3\sum_{i=0}^{\varphi(m)-1}{\mathrm{hw}}(f_{m,p,i,0})~~+~~\sum_{i=0}^{\varphi(m)-1}{\mathrm{hw}}(f_{m,p,i,q})\\
&= q_3\sum_{i=0}^{\varphi(m)-1}{\mathrm{hw}}(f_{m,p,i,0})~~+~~\sum_{i=0}^{p_2-1}{\mathrm{hw}}(f_{m,p,i,q})& \text{by Lemma \ref{lem:trunc}}\\
&= q_3\sum_{i=0}^{\varphi(m)-1}{\mathrm{hw}}(f_{m,p,i,0})~~+~~4q_2+1&  \text{by Lemma~} \ref{lem:rem1} \text{ and Lemma~} \ref{lem:rem2}\\
&= 2q_3\sum_{i=0}^{p_2-2}{\mathrm{hw}}(f_{m,p,i,0})~~+~~4q_2+1&  \text{by Symmetry }\\
&= 2q\left(\sum_{u=0}^{\frac{q_2}{2}-1}(24u+22)~~+\sum_{u=\frac{q_2}{2}}^{q_2-1}(12u+10+6q_{2})\right)+~~4q_2+1&  \text{by Lemma~} \ref{lem:multof31}~ \text{and Lemma~} \ref{lem:not mult31}\\
 &=2q_3(q_{2}(6q_2+7+\frac{9}{2}q_2))+4q_2+1\\
 &=2q_3(q_{2}(7+\frac{21}{2}q_2))+4q_2+1 \\
 &= 2q_3\left(\frac{7}{2} q_{2}(3q_{2}+2)\right)+4q_2+1\\
 &= \frac{7(p_3-2)}{9p_2}(p_{2}-1)(p_2+1)+4\frac{p_2-1}{3}+1& q_3=\frac{p_3-2}{3p_2}, q_2= \frac{p_2-1}{3}\\
 &= (p_3-2)\frac{7(p_2^2-1)}{9p_2}+\frac{4p_2-1}{3}=N(p_3-2)+\frac{4p_2-1}{3}
\end{align*}
Now for $r_3=3p_2-2$, from Theorem \ref{ex:hw}(Theorem 6.1 in \cite{AK}), \[\hw(\Phi_{3p_2p_3})=Np_3-(\frac{4p_2-1}{3}-2N)=N(p_3+2)-\left(\frac{4p_2-1}{3}\right)\]
as desired.\end{proof}
\subsection{Proof of Theorem \ref{thm:r2=2} ($r_2=2$)}\label{sec:r2=2}

\begin{lemma}\label{lem:not multiple of 3 num 1} If $i=3u+v<p_2-1$ where $v=1,2$, then  \[\hw(f_{3p_2,p_3,i,0})=\begin{cases}8(u+1), & \text{if~}0\leq u \leq  \frac{q_2-1}{2}  \\
4(u+1+ \frac{q_2+1}{2} ), & \text{if~} \frac{q_2+1}{2}    \leq u \leq q_2-1\\
\end{cases}\]
\end{lemma}
\begin{proof}As stated in Remark  \ref{rem:1} we have  
\begin{align*}\hw(f_{3p_2,p,i,0})&=\hw(C_i)\\&= \hw\left(\Psi_{3p_2}   
\cdot (x^2-1)\cdot\sum_{j=0}^{\left\lfloor \frac{i}{3} \right\rfloor} x^{6j}\right)& \text{Lemma \ref{lem:tr11}}\\
&=\hw\left(\Phi_3(x)\cdot (x^{p_2}-1)   
(x^2-1)\sum_{j=0}^{\left\lfloor \frac{i}{3} \right\rfloor} x^{6j}\right)\\ &=\hw\left((1+x)(x^{p_2}-1)   
(x^3-1)\sum_{j=0}^{\left\lfloor \frac{i}{3} \right\rfloor} x^{6j}\right)& \Phi_3(x)\cdot(x^2-1)=(x+1)(x^3-1)\\ &=\hw\left((1+x)(1-x^3)\left(\sum_{j=0}^{u} x^{6j}-\sum_{j=0}^{u} x^{6j+p_2}\right)\right)\\
&=2 \cdot  \hw\left(\sum_{j=0}^{u} x^{6j}-\sum_{j=0}^{u} x^{6j+3}-\sum_{j=0}^{u} x^{6j+p_2}+\sum_{j=0}^{u} x^{6j+p_2+3}\right)\end{align*}
If $u \leq \frac{q_2-1}{2}$, then there is no cancellation in the sum above, thus  $\hw(f_{3p_2,p_3,i,0})=2\cdot 4 \cdot (u+1)=8(u+1)$ as desired. On the other hand, if $u>\lceil\frac{q_2}{2}\rceil$, then we have  \begin{align*}\hw(f_{3p_2,p_3,i,0})&=    \hw\left((x+1)(1-x^{p_2})\sum_{j=0}^{u} x^{6j}+(x+1)x^3 (x^{p_2}-1)\sum_{j=0}^{u} x^{6j}\right)\\&=  \hw\left((x+1)(1-x^{p_2})\sum_{j=0}^{u} x^{6j} \right)+\hw \left(x^3(x+1) (x^{p_2}-1)\sum_{j=0}^{u} x^{6j}\right)\\
&= 2 \cdot   \hw\left((x+1)(1-x^{p_2})\sum_{j=0}^{u} x^{6j} \right)\\&=2 \cdot   \left( \hw\left(A_i \right)+ \hw \left(B_i\right)\right)
\end{align*}Where $A_i=A_u=(1-x^{p_2+1})\sum_{j=0}^{u} x^{6j}$ and $B_i=B_u=(x-x^{p_2})\sum_{j=0}^{u} x^{6j} $.
 For $B_u=(x-x^{p_2})\sum_{j=0}^{u} x^{6j} $ there is no cancelation in the sum and we have $\hw(B_u)=2(u+1)$. Assume by contrary that there is a cancelation, then  $6j+1=6k+p_2$ for some $k$ and $j$ which gives us $p_2-1 \equiv 0 \mod 3$ a contradiction to the fact that $p_2\equiv 2 \mod 3.$\\
For $A_u=(1-x^{p_2+1})\sum_{j=0}^{u} x^{6j}$ we claim that $A_u= \sum_{j=0}^{\frac{q_2-1}{2}} x^{6j}~-~x^{p_2+1+6(u-\frac{q_2-1}{2}+j)}$. 
We prove the claim by induction on $u$ starting from $u=\frac{q_2+1}{2}$. \begin{itemize}
\item If $u=\frac{q_2+1}{2}$, then \begin{align*}A_{\frac{q_2+1}{2}}&=(1-x^{p_2+1})\sum_{j=0}^{\frac{q_2+1}{2}} x^{6j}\\
&=(1+x^{6}+\cdots+ x^{3q_2+3})-(x^{p_2+1}+x^{p_2+7}+\cdots+x^{2p_2+2})\\
&=  (1+x^{6}+\cdots+ x^{3q_2-3})-(x^{p_2+7}+\cdots+x^{2p_2+2})\\
&= \sum_{j=0}^{\frac{q_2-1}{2}} x^{6j}~-~x^{p_2+1+6(1+j)} \end{align*}
\item\ Assume that $A_u= \sum_{j=0}^{\frac{q_2-1}{2}} x^{6j}~-~x^{p_2+1+6(u-\frac{q_2-1}{2}+j)}$
\item Consider \begin{align*}A_{u+1}&=(1-x^{p_2+1})\sum_{j=0}^{u+1} x^{6j}\\
                                    &= \sum_{j=0}^{\frac{q_2-1}{2}} x^{6j}~-~x^{p_2+1+6(u-\frac{q_2-1}{2}+j)}+x^{6u+6}(1-x^{p_2+1})\\
                                    &=\sum_{j=0}^{\frac{q_2-1}{2}} x^{6j}-\sum_{j=1}^{\frac{q_2-1}{2}} x^{p_2+1+6(u-\frac{q_2-1}{2}+j)}-x^{6u+6}+x^{6u+6}-x^{p_2+1+6(u+1)}\\
&= \sum_{j=0}^{\frac{q_2-1}{2}} x^{6j}-\sum_{j=0}^{\frac{q_2+1}{2}} x^{p_2+1+6(u+1-\frac{q_2-1}{2}+j)}\\
&= \sum_{j=0}^{\frac{q_2-1}{2}} x^{6j}-\sum_{j=0}^{\frac{q_2-1}{2}} x^{p_2+1+6(u-\frac{q_2-1}{2}+j)}& \text{by reindexing}
   \end{align*}
\end{itemize}
Thus we have $\hw(f_{3p_2,p_3,0})=4(u+1+\left\lceil \frac{q_2}{2} \right\rceil),$ for $\left\lceil \frac{q_2}{2}  \right\rceil  \leq u \leq q_2-1$.\end{proof}

\begin{lemma}\label{lemma:mul2} If $i=3u<p_2-1$ and $p_2\equiv 2 \mod 3$, then \[\hw(f_{3p_2,p,i,0})=\begin{cases}6+8u, & \text{if~}0\leq u\leq \frac{q_2-1}{2}  \\

4u+5+2q_2 , &\text{if}~ \frac{q_2+1}{2}\leq u\leq q_2-1 \\
\end{cases}\]
\end{lemma}
\begin{proof} In this case the cyclotomic polynomial $\Phi_{3p_2p_3}$ may have some non flat coefficients, so we will worry about cancelations and overlapping.
\begin{align*}\hw(f_{3p_2,p,i,0})&= \hw\left(\Psi_{3p_2}   
\cdot ((x^2-1)\sum_{j=0}^{\left\lfloor \frac{i-1}{3} \right\rfloor} x^{6j}~-x^{2i})\right)& \text{from Lemma \ref{lem:tr11}}\\
 &=\hw\left((1+x)(x^{p_2}-1)   
(x^3-1)\sum_{j=0}^{\left\lfloor \frac{i-1}{3} \right\rfloor} x^{6j}~-x^{2i}\Psi_{3p_2}\right)& \\ 
&=\hw\left((1+x)(x^{p_2}-1)   
(x^3-1)\sum_{j=0}^{u-1} x^{6j}~-x^{6u}\Psi_{3p_2}\right)& \\ 
\end{align*}We have the following cases \begin{enumerate}
\item 
$u\leq\frac{q_2-1}{2}, $ in this case there are no cancelations or overlapping in the above  sum. Thus $\hw(f_{3p_2,p,i,0})=8u+6$
\item $u\geq \frac{q_2+1}{2}$. Notice that $\Psi_{3p_2}\cdot (x^2-1)=(x+1)(x^3-1)(1-x^{p_2})$, so we have
\[\hw(f_{3p_2,p,i,0})=  \hw \left((x+1)A_u+B_u\right)\] where  $A_u=(x^3-1)(1-x^{p_2})\sum_{j=0}^{u-1}x^{6j}$and $B_u=x^{6u}\Psi_{3p_2}$.  We claim that  \[(x+1)A_u=(x+1)\sum_{j=0}^{q_2}(-1)^{j+1}x^{3j}+ (1-x^2)\sum_{j=q_2}^{2u-2}(-1)^{j+1}x^{3j+2}+(x+1)\sum_{j=2u-1}^{2u+q_2-1}(-1)^jx^{3j+2}\]
The proof of this claim is similar to the one in the proof of Lemma  \ref{lem:multof31}, so we omit it. Now since there is a cancelation of the term $x^{6u+2}$ between $B_u$ and $(x+1)A_u$ and an overlapping between the term $x^{6u}$ between $B_u$ and $(x+1)A_u$, we have $\hw(f_{3p_2,p,i,0})=  \hw \left((x+1)A_u+B_u\right)=2(q_2+1)+2(2u-1-q_2)+2(q_2+1)-1+4=4u+2q_2+5.$
\end{enumerate}
\end{proof}
\begin{lemma}\label{lem:rem22} If $i=3u+v<p_2-1$ where $v=1,2$ and $p_2\equiv 2 \mod 3$, then\[\hw(f_{3p_2,p,i,q})=\begin{cases}1, &\text{if~}0\leq u\leq  \frac{q_2-3}{2} \\
v, &\text{if~}  \frac{q_2-1}{2} \leq u

\end{cases}\]
\end{lemma}
\begin{proof}Similar to the proof of Lemma \ref{lem:rem1} so we move it to the appendix.
\end{proof}
\begin{theorem}Let $3<p_2<p_3$ be odd prime numbers such that $p_2\equiv 2 \mod 3$. Then 
\[\hw(\Phi_{3p_2p_3})= \begin{cases}N(p_3-2)+\frac{4(p_2+1)}{3}, & \text{if}~~ r_3=2 \\
N(p_3+2)-\frac{4(p_2+1)}{3}, &\text{if}~~ r_3=p_2-2 \\
\end{cases}\] where $N=\frac {(p_2+1) ({7p_2-2})}{9p_2}. $
\end{theorem}
\begin{proof}Note if $r_2=r_3=2$, then 
\begin{align*}\hw(\Phi_{3p_2p_3})&=  \sum_{\substack{0\leq i\leq\varphi
(m)-1}}{\mathrm{hw}}(f_{m,p,i})\ \ x^{ip_{3}}\\
&  =q_3\sum_{i=0}^{\varphi(m)-1}{\mathrm{hw}}(f_{m,p,i,0})~~+~~\sum_{i=0}^{\varphi(m)-1}{\mathrm{hw}}(f_{m,p,i,q})\\
&= q_3\sum_{i=0}^{\varphi(m)-1}{\mathrm{hw}}(f_{m,p,i,0})~~+~~\sum_{i=0}^{p_2-1}{\mathrm{hw}}(f_{m,p,i,q})&\text{by Lemma \ref{lem:trunc}}\\
&= q_3\sum_{i=0}^{\varphi(m)-1}{\mathrm{hw}}(f_{m,p,i,0})~~+~~4q_2+3&  \text{by Lemma} ~\ref{lem:rem22}\\
&= 2q_3\sum_{i=0}^{p_2-2}{\mathrm{hw}}(f_{m,p,i,0})~~+~~4q_2+3&  \text{by Symmetry }\\
&= 2q_3\left(\sum_{u=0}^{\frac{q_2-1}{2}}(24u+22)\right)~~+4q_2+3 & \text{by Lemma}~ \ref{lem:not multiple of 3 num 1}~ \text{and} ~\ref{lemma:mul2}\\&+2q_3\left(\sum_{u=\frac{q_2+1}{2}}^{q_2-1}(12u+6q_{2}+17)\right)+2q_3(2p_2+1)& \text{by Lemma \ref{lem:p2-1}}\\&~~~~~~&  \\
 &=2q_3(\frac{21}{2}q_2^2+\frac{21}{2}q_2+1)+2q_3(2p_2+1) +4q_2+3\\
 &= 21q_{2}q_3(q_2+1)+2q_3(2p_2+2) +4q_2+3\\
 &=7 \frac{p_3-2}{p_2} \cdot \frac{p_2+1}{3} \cdot \frac{p_2-2}{3}+4\frac{p_3-2}{3p_2}(p_2+1)+4\frac{p_2-2}{3}+3 \\
 &= \frac{p_3-2}{9p_2} \cdot (p_2+1) ({7p_2-2})+\frac{4p_2+1}{3}\\
 &=N(p_3-2)+\frac{4p_2+1}{3}
\end{align*} as desired. Now for $r_3=3p_2-2$, from Theorem \ref{ex:hw} (Theorem 6.1 in \cite{AK}), \[\hw(\Phi_{3p_2p_3})=Np_3-(\frac{4(p_2+1)}{3}-2N)=N(p_3+2)-\frac{4(p_2+1)}{3}\]
as desired.
\end{proof}

\bigskip

\noindent \textbf{Acknowledgment:} We would like to thank  Hoon Hong and Eunjeong Lee for   
the discussions lead to the proof of Lemma \ref{l:block}.
\bigskip

\appendix                                     
\section{Technical proofs}\label{A1}\begin{itemize}
\item \textbf{Proof of Lemma \ref{lemma:Phi3p}:} We have\begin{align*}
\Phi_{3p_{2}}\left(  x\right)   &  =\sum_{i\geq0}f_{3,p_{2},i}\left(  x\right)
\ x^{ip_{2}} &  &  \text{where }\deg f_{m,p_{2},i}\left(  x\right)  <p_{2}\\
f_{3,p_{2},i}\left(  x\right)   &  =\sum_{j\geq0}f_{3,p_{2},i,j}\left(  x\right)
x^{3j} &  &  \text{where }\deg f_{3,p_{2},i,j}\left(  x\right)  <3
\end{align*} It is clear that $i=0,1.$
Using  Lemma \ref{l:block}-2 and knowing that $a_0=a_1=a_2=2$ and $b_1=1,b_0=-1,$ we can compute the coefficient  $c_k$ in  $f_{3,p_2,i,j}$:
\[\begin{array}{c|c|c|c|c}
i  & c_0 & c_1& c_2 & ~ f_{3,p_2,i,0}  \\\hline
0 &1 &-1&0& 1-x \\
1&0  &-1  &1&x^2-x\\
\end{array}\]Thus, \[\Phi_{3p_{2}}\left(  x\right)     =\sum_{j=0}^{q_2-1}(1-x)x^{3j}+x^{3q_2}(1-x^{r_2-1})+
\ x^{p_{2}}\sum_{j=0}^{q_2-1}(x-x^2)x^{3j}+x^{\varphi(3p_2)}\] which proves the Lemma.
\begin{remark} The last proof is a concrete example which illustrate the partition on cyclotomic polynomials which we use in the proofs of this paper.
\end{remark}
\item 
\textbf{Proof of Lemma \ref{lem:rem22}:}
It is enough to prove that only one of  the terms $x^0$ and $x^1$ will appear in $f_{3p_2,p,i,0}$ when $0\leq u \leq \frac{q_2-3}{2}$ and only one of the terms $x^0$ and $x$ will appear in  $f_{3p_2,p,i,0}$ when $ \frac{q_2-1}{2}\leq u \leq q_2$ if $v=1$ while both of them will appear when $v=2.$ In the table below we list the term appears in $C_i=\Psi_{3p_2}   
\cdot (x^2-1)\cdot\sum_{j=0}^{  u} x^{6j}$ and the corresponding term in $f_{3p_2,p,i,0}$.

\[\begin{array}{c|c|c|c|c}
u/\text{Terms in }  & C_i  ~~ \text{if}~~ v=1 & ~ f_{3p_2,p,i,0} ~\text{if } v=1&  C_i ~~ \text{if}~~ v=2 & ~ f_{3p_2,p,i,0} ~\text{if } v=2 \\\hline
0\leq u \leq \frac{q_2-1}{2} & -x^{6u+2v+1} &-x&x^{6u+2v}& 1 \\
\frac{q_2+1}{2} \leq u \leq q_2& x^{6u+2v} & 1 &-x^{6u+2v}-x^{6u+2v+1}&1+x\\
\end{array}\] The reasoning  why the term $x^{6u+1}$ will not appear in $C_i$ for $\frac{q_2}{2} \leq u \leq q_2$, was given in the proof of Lemma \ref{lem:multof31}.  Now for the case $u > \frac{q_2}{2}-1$ since $ i=3u+v$ we have $2i>6(\frac{q_2}{2}-1)+2v=3q_2-6+2v=p_2-7+v.$ For $x^0$ the equation $x^{-2i}\cdot x^{w}=x^0\Rightarrow w=2i=6u+2v$ which has two solutions when $v=1$ and one solution when $v=0$.  
\end{itemize}
\section{Examples}\label{A2} Here we add two examples to illustrate the results in the sections \ref{sec:r2=1} and \ref{sec:r2=2}. \begin{example} We will illustrate the results in subsection  \ref{sec:r2=1} by an  example. Let $p_{1}=3, p_{2}=7$ and $p_3=23.$ \ Note that $r_2=1$ and $r_3=2$. The following two figures present the relationship between $i$, $\hw(f_{21,p_3,i,0})$ and $\hw(f_{21,p_3,i,q})$.
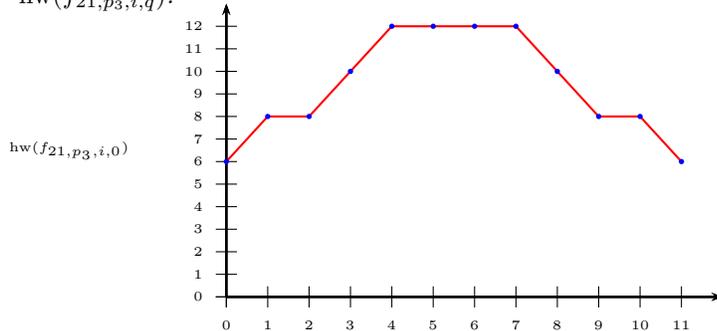
\begin{center}\tiny
\psset{xunit=0.55cm,yunit=0.30cm}
\begin{pspicture}(20,12)(0,0)
    \savedata{\mydatasinexp}%
[{0,6},{1,8},{2,8},{3,10},{4,12},{5,12},{6,12},{7,12},{8,10},{9,8},{10,8},{11,6}]
    \psaxes{->}(0,0)(0,0)(12.,13.)
    \dataplot[plotstyle=line,linecolor=red]
      {\mydatasinexp}
    \dataplot[plotstyle=dots,linecolor=blue,
      dotsize=2pt]{\mydatasinexp}
\rput(10,-2.5){$i$}
\rput(-3.8,6.5){$\hw(f_{21,p_3,i,0})$}
\end{pspicture}

\tiny
\psset{xunit=0.50cm,yunit=0.45cm}
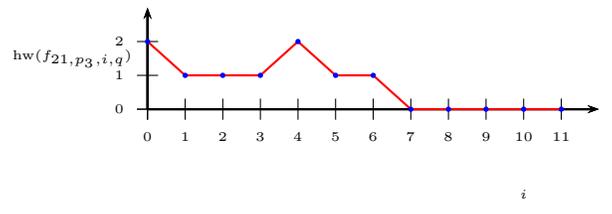
\begin{pspicture}(-13,-13)(0,-1)
    \savedata{\mydatasinexp}%
[{0,2},{1,1},{2,1},{3,1},{4,2},{5,1},{6,1},{7,0},{8,0},{9,0},{10,0},{11,0}]
    \psaxes{->}(0,0)(0,0)(12.,3.)
    \dataplot[plotstyle=line,linecolor=red]
      {\mydatasinexp}
    \dataplot[plotstyle=dots,linecolor=blue,
      dotsize=2pt]{\mydatasinexp}
\rput(10,-2.5){$i$}
\rput(-2,1.5){$\hw(f_{21,p_3,i,q})$}
\end{pspicture}
\end{center}
\end{example} \newpage
\begin{example} We will illustrate the results in subsection  \ref{sec:r2=2}  by an  example. Let $p_{1}=3, p_{2}=11$ and $p_3=101.$ \ Note that $r_2=2$ and $r_3=2$. The following figure presents the relationship between $i$ and $\hw(f_{33,p_3,i,0})$. 
\begin{center}\tiny
\psset{xunit=0.55cm,yunit=0.45cm}
\begin{pspicture}(16,24)(-1,0)
    \savedata{\mydatasinexp}%
[{0,6},{1,8},{2,8},{3,14},{4,16},{5,16},{6,19},{7,20},{8,20},{9,23},{10,23},{11,20},{12,20},{13,19},{14,16},{15,16},{16,14},{17,8},{18,8},{19,6}]
    \psaxes{->}(0,0)(0,0)(20.,24.)
    \dataplot[plotstyle=line,linecolor=red]
      {\mydatasinexp}
    \dataplot[plotstyle=dots,linecolor=blue,
      dotsize=2pt]{\mydatasinexp}
\rput(10,-2.5){$i$}
\rput(-3.8,6.5){$\hw(f_{33,p_3,i,0})$}
\end{pspicture} \\
\tiny
\psset{xunit=0.55cm,yunit=0.45cm}
\begin{pspicture}(17,5)(-1,0)
    \savedata{\mydatasinexp}%
[{0,2},{1,1},{2,1},{3,2},{4,1},{5,2},{6,1},{7,1},{8,1},{9,1},{10,1},{11,0},{12,0},{13,0},{14,0},{15,0},{16,0},{17,0},{18,0},{19,0}]
    \psaxes{->}(0,0)(0,0)(20.,3.)
    \dataplot[plotstyle=line,linecolor=red]
      {\mydatasinexp}
    \dataplot[plotstyle=dots,linecolor=blue,
      dotsize=2pt]{\mydatasinexp}
\rput(10,-2.5){$i$}
\rput(-3.8,0.5){$\hw(f_{33,p_3,i,q})$}
\end{pspicture}\end{center}
\end{example} 
\bibliographystyle{unsrt}
\bibliography{allrefs}
{}
\end{document}